\documentclass[letterpaper, 10 pt, conference]{ieeeconf}  
\IEEEoverridecommandlockouts                              % This command is only
\usepackage{epsfig}
\usepackage{amsfonts}
\usepackage{amsmath}
\usepackage{amssymb}

\newcommand{\bmtx}{\begin{bmatrix}}
\newcommand{\emtx}{\end{bmatrix}}
\newcommand{\bsmtx}{\left[ \begin{smallmatrix}} 
\newcommand{\esmtx}{\end{smallmatrix} \right]} 

\newcommand{\field}[1]{\mathbb{#1}}
\newcommand{\R}{\field{R}}

\newcommand{\N}{\field{N}}
\newtheorem{lemma}{Lemma}
\newtheorem{theorem}{Theorem}

\newcommand{\nCk}[2]{ \left( \begin{smallmatrix} #1 \\ #2
    \end{smallmatrix} \right) } 

\newcommand{\cA}{{\cal A}}
\newcommand{\cV}{{\cal V}}
\newcommand{\cZ}{{\cal Z}}
\newcommand{\cS}{{\cal S}}
\newcommand{\cI}{{\cal I}}

\newcommand{\Rn}{\R[x]}
\newcommand{\sos}{{\Sigma[x]}}
\newcommand{\dv}{u}

\newcommand{\ch}{\mbox{convhull}}

\title{\LARGE \bf
Simplification Methods for Sum-of-Squares Programs
}

\author{Peter Seiler% <-this % stops a space
\thanks{P. Seiler is with Aerospace and Engineering Mechanics Department, 
        University of Minnesota, {\tt\small seiler@aem.umn.edu}}, %
Qian Zheng%
\thanks{Q. Zheng is with Aerospace and Engineering Mechanics Department, 
        University of Minnesota, {\tt\small qzheng@aem.umn.edu}}, %
and Gary J. Balas
\thanks{G.J. Balas is with Aerospace and Engineering Mechanics Department, 
        University of Minnesota, {\tt\small balas@umn.edu}}%
}

\begin{document}

\maketitle
\thispagestyle{empty}
\pagestyle{empty}

%-------------------------------------------------------
% Abstract
%-------------------------------------------------------
\begin{abstract}
  
  A sum-of-squares is a polynomial that can be expressed as a sum of
  squares of other polynomials.  Determining if a sum-of-squares
  decomposition exists for a given polynomial is equivalent to a
  linear matrix inequality feasibility problem.  The computation
  required to solve the feasibility problem depends on the number of
  monomials used in the decomposition.  The Newton polytope is a
  method to prune unnecessary monomials from the decomposition.  This
  method requires the construction of a convex hull and this can be
  time consuming for polynomials with many terms.  This paper presents
  a new algorithm for removing monomials based on a simple property of
  positive semidefinite matrices.  It returns a set of monomials that
  is never larger than the set returned by the Newton polytope method
  and, for some polynomials, is a strictly smaller set. Moreover, the
  algorithm takes significantly less computation than the convex hull
  construction.  This algorithm is then extended to a more general
  simplification method for sum-of-squares programming.

\end{abstract}

%-------------------------------------------------------
% Introduction
%-------------------------------------------------------
\section{Introduction}
\label{sec:intro}

A polynomial is a sum-of-squares (SOS) if it can be expressed as a sum
of squares of other polynomials.  There are close connections between
SOS polynomials and positive semidefinite matrices
\cite{chesi99,chesi09,choi95,powers98,parrilo00,lasserre01,parrilo03}.
For a given polynomial the search for an SOS decomposition is
equivalent to a linear matrix inequality feasibility problem.  It is
also possible to formulate optimization problems with polynomial
sum-of-squares constraints \cite{parrilo00,parrilo03}.  There is
freely available software that can be used to solve these SOS
feasibility and optimization problems
\cite{sostools04,yalmip04,nasaworkshop09,henrion09}.  Many nonlinear
analysis problems, e.g. Lyapunov stability analysis, can be formulated
within this optimization framework
\cite{parrilo00,parrilo03,tan06,topcu08b}.

Computational growth is a significant issue for these optimization
problems. For example, consider the search for an SOS decomposition:
given a polynomial $p$ and a vector of monomials $z$, does there exist
a matrix $Q\succeq 0$ such that $p = z^T Q z$?  The computation
required to solve the corresponding linear matrix inequality
feasibility problem grows with the number of monomials in the vector
$z$.  The Newton polytope \cite{reznick78,sturmfels98} is a method to
prune unnecessary monomials from the vector $z$.  This method is
implemented in SOSTOOLs \cite{sostools04}. One drawback is that this
method requires the construction of a convex hull and this
construction itself can be time consuming for polynomials with many
terms.

This paper presents an alternative monomial reduction method called
the zero diagonal algorithm.  This algorithm is based on a simple
property of positive semidefinite matrices: if the $(i,i)$ diagonal
entry of a positive semidefinite matrix is zero then the entire
$i^{th}$ row and column must be zero.  The zero diagonal algorithm
simply searches for diagonal entries of $Q$ that are constrained to be
zero and then prunes the corresponding monomials.  This algorithm can
be implemented with very little computational cost using the Matlab
\texttt{find} command.  It is shown that final list of monomials
returned by the zero diagonal algorithm is never larger than the
pruned list obtained from the Newton polytope method.  For some
problems the zero diagonal algorithm returns a strictly smaller set of
monomials.  Results contained in this paper are similar to and
preceded by those found in the prior work \cite{lofberg09,waki09}.

The basic idea in the zero diagonal algorithm is then extended to a
more general simplification method for sum-of-squares programs.  The
more general method also removes free variables that are implicitly
constrained to be equal to zero.  This can improve the numerical
conditioning and reduce the computation time required to solve the SOS
program.  Both the zero diagonal elimination algorithm and the
simplification procedure for SOS programs are implemented in SOSOPT
\cite{nasaworkshop09}. 

%% The remainder of the paper has the following structure. The next
%% section briefly reviews background material on sum-of-squares
%% polynomials. The Newton polytope method for monomial reduction is
%% described in Section~\ref{sec:newton}.  Section~\ref{sec:zda}
%% describes the zero diagonal algorithm and proves that it returns a set
%% of monomials that is a subset of those returned by the Newton polytope
%% method.  A more general simplification procedure for SOS programs is
%% then provided in Section~\ref{sec:sossimp}. Finally, conclusions are
%% given in Section~\ref{sec:conclusions}.

%-------------------------------------------------------
% SOS Polys 
%-------------------------------------------------------
\section{SOS Polynomials}
\label{sec:sospoly}

$\N$ denotes the set of nonnegative integers, $\{0, 1, \ldots\}$, and
$\N^n$ is the set of $n$-dimensional vectors with entries in $\N$. For
$\alpha \in \N^n$, a \underline{monomial} in variables
$\{x_1,\ldots,x_n\}$ is given by $x^\alpha:= x_1^{\alpha_1}
x_2^{\alpha_2} \cdots x_n^{\alpha_n}$.  $\alpha$ is the degree vector
associated with the monomial $x^\alpha$. The degree of a monomial is
defined as $\deg{x^\alpha} := \sum_{i=1}^n \alpha_i$.  A
\underline{polynomial} is a finite linear combination of monomials:
\begin{align} 
\label{eq:poly}
p := \sum_{\alpha \in \cA}
  c_\alpha x^\alpha = \sum_{\alpha \in \cA} c_\alpha x_1^{\alpha_1}
  x_2^{\alpha_2} \cdots x_n^{\alpha_n}
\end{align}
where $c_\alpha \in \R$, $c_\alpha \ne 0$, and $\cA$ is a finite
collection of vectors in $\N^n$.  $\Rn$ denotes the set of all
polynomials in variables $\{x_1, \ldots, x_n\}$ with real
coefficients.  Using the definition of $\deg$ for a monomial, the
degree of $p$ is defined as $\deg{p}:= \max_{\alpha \in \cA} \ 
\ [ \deg{x^\alpha} ]$.

A polynomial $p$ is a \underline{sum-of-squares} (SOS) if there exist
polynomials $\{f_i\}_{i=1}^m$ such that $p = \sum_{i=1}^m f_i^2$.  The
set of SOS polynomials is a subset of $\Rn$ and is denoted by $\sos$.
If $p$ is a sum-of-squares then $p(x) \ge 0$ $\forall x \in \R^n$.
However, non-negative polynomials are not necessarily SOS 
\cite{reznick00}.

Define $z$ as the column vector of all monomials in variables $\{x_1,
\ldots, x_n\}$ of degree $\le d$: \footnote{Any ordering of the
  monomials can be used to form $z$. In Equation~\ref{eq:monomvec},
  $x^\alpha$ precedes $x^\beta$ in the definition of $z$ if
  $\deg{x^\alpha} < \deg{x^\beta}$ OR $\deg{x^\alpha} =
  \deg{x^\beta}$ and the first nonzero entry of $\alpha-\beta \mbox{
    is } > 0$. } 
{\small
\begin{align}
\label{eq:monomvec}
z:=  \bmtx 1, \ x_1, \ x_2, \ \ldots, \ x_n, 
   \ x_1^2, \ x_1 x_2, \ \ldots, \ x_n^2, \ \ldots, \ x_n^d  \emtx^T 
\end{align}
} 
There are $\nCk{k+n-1}{k}$ monomials in $n$ variables of degree $k$.
Thus $z$ is a column vector of length $l_z := \sum_{k=0}^d
\nCk{k+n-1}{k} = \nCk{n+d}{d}$.  If $f$ is a polynomial in $n$
variables with degree $\le d$ then by definition $f$ is a finite
linear combination of monomials of degree $\le d$.  Consequently,
there exists $a \in \R^{l_z}$ such that $f=a^T z$.  

Two useful facts from \cite{reznick78} are:
\begin{enumerate} 
\item If $p$ is a sum-of-squares then $p$ must have even degree.
\item If $p$ is degree $2d$ ($d \in \N$) and $p = \sum_{i=1}^m f_i^2$
  then $\deg{f_i} \le d$  $\forall i$.
\end{enumerate}

The following theorem, introduced as the ``Gram Matrix'' method by
\cite{choi95,powers98}, connects SOS polynomials and positive
semidefinite matrices.

\vspace{0.1in}
\begin{theorem}
\label{thm:gram}
Let $p \in \Rn$ be a polynomial of degree $2d$ and $z$ be the $l_z
\times 1$ vector of monomials defined in Equation~\ref{eq:monomvec}.
Then $p$ is a SOS if and only if there exists a symmetric matrix $Q
\in \R^{l_z \times l_z}$ such that $Q\succeq 0$ and $p = z^T Q z$.
\end{theorem} 
\begin{proof}
  ($\Rightarrow$) If $p$ is a SOS, then there exists polynomials
  $\{f_i\}_{i=1}^m$ such that $p = \sum_{i=1}^m f_i^2$.  By fact 2
  above, $\deg{f_i} \le d$ for all $i$.  Thus, for each $f_i$ there
  exists a vector, $a_i \in \R^{l_z}$, such that $f_i = a_i^T z$.
  Define the matrix $A\in \R^{l_z \times m}$ whose $i^{th}$ column
  is $a_i$ and define $Q:= A A^T \succeq 0$. Then $p = z^T Q z$.
  
  ($\Leftarrow$) Assume there exists $Q=Q^T \in \R^{l_z \times l_z}$
  such that $Q \succeq 0$ and $p = z^T Q z$.  Define $m:=
  rank(Q)$.  There exists a matrix $A \in \R^{l_z \times m}$ such that
  $Q = A A^T$. Let $a_i$ denote the $i^{th}$ column of $A$ and define
  the polynomials $f_i := z^T a_i$. By definition of $f_i$, $p =
  z^T (AA^T) z = \sum_{i=1}^m f_i^2$.
\end{proof}
\vspace{0.1in}

%\emph{Proof:}  
%$\blacksquare$

Determining if an SOS decomposition exists for a given polynomial $p$
is equivalent to a feasibility problem:
\begin{align}
  \label{eq:sosfeas}
  \mbox{Find } Q\succeq 0 \mbox{ such that } p=z^TQz
\end{align}
$Q$ is constrained to be positive semi-definite and equating
coefficients of $p$ and $z^TQz$ imposes linear equality constraints on
the entries of $Q$. Thus this is a linear matrix inequality (LMI)
feasibility problem.  There is software available to solve for SOS
decompositions \cite{sostools04,yalmip04,nasaworkshop09}.  These
toolboxes convert the SOS feasibility problem to an LMI problem.  The
LMI problem is then solved with a freely available LMI solver, e.g.
Sedumi \cite{sedumi99}, and an SOS decomposition is constructed if a
feasible solution is found.  These software packages also solve SOS
synthesis problems where some of the coefficients of the polynomial
are treated as free variables to be computed as part of the
optimization.  These more general SOS optimization problems are
discussed further in Section~\ref{sec:sossimp}.  Many analysis problems
for polynomial dynamical systems can be posed within this SOS
synthesis framework \cite{parrilo00,parrilo03,tan06,topcu08b}.

%-------------------------------------------------------
% Section
%-------------------------------------------------------
\section{Newton Polytope}
\label{sec:newton}

As discussed in the previous section, the search for an SOS
decomposition is equivalent to an LMI feasibility problem.  One issue
is that the computational complexity of this LMI feasibility problem
grows with the dimension of the Gram matrix.  For a polynomial of
degree $2d$ in $n$ variables there are, in general, $l_z =
\nCk{n+d}{d}$ monomials in $z$ and the Gram matrix $Q$ is $l_z \times
l_z$.  $l_z$ grows rapidly with both the number of variables and the
degree of the polynomial.  However, any particular polynomial $p$ may
have an SOS decomposition with fewer monomials. The Newton Polytope
\cite{reznick78,sturmfels98} is an algorithm to reduce the dimension
$l_z$ by pruning unnecessary monomials from $z$.

First, some terminology is provided regarding polytopes
\cite{bronsted83,grunbaum03}.  For any set $\cA \subseteq\R^n$,
$\ch(A)$ denotes the convex hull of $\cA$. Let $C \subseteq \R^n$ be a
convex set.  A point $\alpha \in C$ is an extreme point if it does not
belong to the relative interior of any segment $[\alpha_1,\alpha_2]
\subset C$. In other words, if $\exists \alpha_1, \alpha_2 \in C$ and
$0 < \lambda < 1$ such that $\alpha=\lambda \alpha_1 +
(1-\lambda)\alpha_2$ then $\alpha_1 = \alpha_2 = \alpha$.  A convex
polytope (or simply polytope) is the convex hull of a non-empty,
finite set $\{\alpha_1,\ldots,\alpha_p\} \subseteq \R^n$.  The extreme
points of a polytope are called the vertices.  Let $C$ be a polytope
and let $\cV$ be the (finite) set of vertices of $C$.  Then
$C=\ch(\cV)$ and $\cV$ is a minimal vertex representation of $C$.  The
polytope $C$ may be equivalently described as an intersection of a
finite collection of halfspaces, i.e. there exists a matrix $H\in
\R^{N\times n}$ and a vector $g\in\R^N$ such that $C=\{ \alpha\in \R^n
\ : \ H\alpha \le g\}$.  This is a facet or half-space representation
of $C$.

The \underline{Newton Polytope} (or cage) of a polynomial
$p = \sum_{\alpha \in \cA} c_\alpha x^\alpha$ is defined as $C(p):=
\ch( \cA )\subseteq \R^n$
\cite{reznick78}.  The reduced Newton polytope is $\frac{1}{2}C(p):=\{
\frac{1}{2} \alpha \ : \ \alpha \in C(p) \}$. The following theorem
from \cite{reznick78} is a key result for monomial reduction.
\footnote{ A polynomial $p$ is a \underline{form} if all monomials
  have the same degree.  The results in \cite{reznick78} are stated
  and proved for forms.  A given polynomial can be converted to a form
  by adding a single dummy variable of appropriate degree to each
  monomial. The results in \cite{reznick78} apply to polynomials by
  this homogenization procedure.}

\vspace{0.1in}
\begin{theorem}
\label{thm:newton}
If $p = \sum_{i=1}^m f_i^2$ then the vertices of $C(p)$ are vectors
whose entries are even numbers and $C(f_i) \subseteq \frac{1}{2}
C(p)$.
\end{theorem}
\vspace{0.1in}

This theorem implies that any monomial $x^\alpha$ appearing in the
vector $z$ of an SOS decomposition $z^T Q z$ must satisfy $\alpha \in
\frac{1}{2}C(p) \cap \N^n$.  This forms the basis for the Newton
polytope method for pruning monomials: Let $p$ be a given polynomial
of degree $2d$ in $n$ variables with monomial degree vectors specified
by the finite set $\cA$. First, create the $l_z\times 1$ vector $z$
consisting of all monomials of degree $\le d$ in $n$ variables. There
are $l_z = \nCk{n+d}{d}$ monomials in this complete list. Second,
compute a half-space representation $\{ \alpha\in \R^n \ : \ H\alpha
\le g\}$ for the reduced Newton polytope $\frac{1}{2}C(p)$.  Third,
prune out any monomials in $z$ that are not elements of
$\frac{1}{2}C(p)$.  This algorithm is implemented in SOSTOOLs
\cite{sostools04}.  The third step amounts to checking each monomial
in $z$ to see if the corresponding degree vector satisfies the
half-plane constraints $H\alpha \le g$.  This step is computationally
very fast.  The second step requires computing a half-plane
representation for the convex hull of $\frac{1}{2} \cA$.  This can be
done in Matlab, e.g.  with \texttt{convhulln}.  However, this step can
be time-consuming when the polynomial has many terms ($\cA$ has many
elements). The next section provides an alternative implementation of
the Newton Polytope algorithm that avoids constructing the half-space
representation of the reduced Newton polytope.

\vspace{0.1in}
\underline{\emph{Example:}} Consider the following polynomial
\begin{align}
\label{eq:polyex}
p = 3x_1^4 - 2x_1^2 x_2 + 7x_1^2 - 4x_1x_2 + 4x_2^2 + 1
\end{align}
$p$ is a degree four polynomial in two variables. The list of
all monomials in two variables with degree $\le 2$ is:
\begin{align}
\label{eq:fullz}
z = \bmtx 1 & x_1 & x_2 & x_1^2 & x_1 x_2 & x_2^2 \emtx^T
\end{align}
The length of $z$ is $l_z = 6$. An SOS decomposition of a degree four
polynomial would, in general, include all six of these monomials. The
Newton Polytope can be used to prune some unnecessary monomials in
this list.

The set of monomial degree vectors for $p$ is $\cA:=\{ \bsmtx 4 \\ 0
\esmtx, \ \bsmtx 2 \\ 1 \esmtx, \ \bsmtx 2 \\ 0 \esmtx, \ \bsmtx 1 \\
1 \esmtx, \ \bsmtx 0 \\ 2 \esmtx, \ \bsmtx 0 \\ 0 \esmtx \}$.  These
vectors are shown as circles in Figure~\ref{fig:newton}.  The Newton
Polytope $C(p)$ is the large triangle with vertices $\{ \bsmtx 4 \\ 0
\esmtx, \ \bsmtx 0 \\ 0 \esmtx, \ \bsmtx 0 \\ 2 \esmtx\}$.
Figure~\ref{fig:rednewton} shows the degree vectors for the six
monomials in $z$ (circles) and the reduced Newton polytope (large
triangle).  The reduced Newton polytope $\frac{1}{2} C(f)$ is the
triangle with vertices $\{ \bsmtx 2 \\ 0 \esmtx, \ \bsmtx 0 \\ 0
\esmtx, \ \bsmtx 0 \\ 1 \esmtx \}$.  By Theorem~\ref{thm:newton},
$x_1x_2$ and $x_2^2$ can not appear in any SOS decomposition of $p$
because $\bsmtx 1 \\ 1 \esmtx, \bsmtx 0 \\ 2 \esmtx \notin\frac{1}{2}
C(f)$. These monomials can be pruned from $z$ and the search for an
SOS decomposition can be performed using only the four monomials in
the reduced Newton polytope:
\begin{align}
  z = \bmtx 1 & x_1 & x_2 & x_1^2 \emtx^T
\end{align}
The length of the reduced vector $z$ is $l_z = 4$.  The SOS feasibility
problem with this reduced vector $z$ (Equation~\ref{eq:sosfeas}) is
feasible.  The following matrix is one feasible solution:
\begin{align}
  Q = \bsmtx 1 & 0 & 0 & 0 \\ 0 & 7 & -2 & 0 \\ 
            0 & -2 & 4 & -1 \\ 0 & 0 & -1 & 3 \esmtx
\end{align}
$p$ is SOS since $p=z^TQz$ and $Q\succeq 0$.

\begin{figure}[h]
\centerline{\psfig{figure=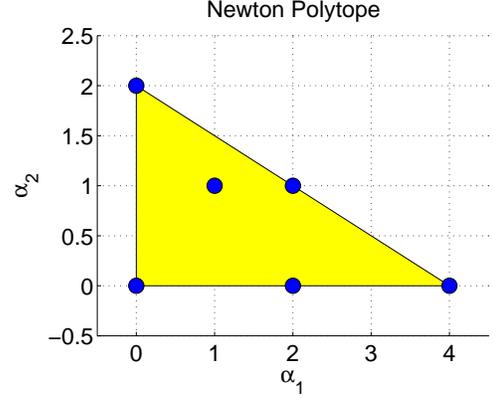,width=2.8in,angle=0}}
\caption{Newton polytope (large triangle) and monomial degree vectors (circles)}
\label{fig:newton}
\end{figure}

\begin{figure}[h]
\centerline{\psfig{figure=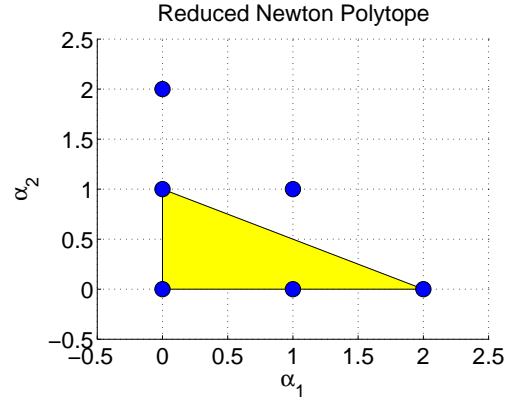,width=2.8in,angle=0}}
\caption{Reduced Newton polytope (large triangle) and degree vectors for all monomials of degree = 0, 1, 2 (circles)}
\label{fig:rednewton}
\end{figure}

%-------------------------------------------------------
% Section
%-------------------------------------------------------
\section{Zero Diagonal Algorithm}
\label{sec:zda}

The zero diagonal algorithm searches for diagonal entries of the Gram
matrix that are constrained to be zero and then prunes the associated
monomials from $z$.  The remainder of the section describes this
algorithm in more detail.

As mentioned in Section~\ref{sec:sospoly}, equating the coefficients
of $p$ and $z^TQz$ leads to linear equality constraints on the entries
of $Q$. The structure of these equations plays an important role in
the proposed algorithm. Let $z$ be the $l_z\times 1$ vector of all
monomials in $n$ variables of degree $\le d$
(Equation~\ref{eq:monomvec}). Define the corresponding set of degree
vectors as $M := \{ \alpha_1,\ldots, \alpha_{l_z}\} \subseteq \N^n$.
$z^TQz$ is a polynomial in $x$ with coefficients that are linear
functions of the entries of $Q$:
\begin{align}
\label{eq:zQz1}
  z^TQz = \sum_{i=1}^m \sum_{j=1}^m Q_{i,j}
  x^{\alpha_i+\alpha_j}
\end{align}
The entries of $z$ are not independent: it is possible that $z_i z_j =
z_k z_l$ for some $i,j,k,l \in \{1,\ldots,l_z\}$.  The unique degree
vectors in Equation~\ref{eq:zQz1} are given by the set
\begin{align}
M+M:=\{\alpha \in \N^n : \exists \alpha_i,\alpha_j \in M \mbox{ s.t. }
   \alpha = \alpha_i + \alpha_j \}
\end{align}
The polynomial $z^TQz$ can be rewritten as:
\begin{align}
  z^TQz = \sum_{\alpha \in M+M}
     \left( \sum_{(i,j)\in S_\alpha} Q_{i,j} \right) x^\alpha
\end{align}
where $S_\alpha := \{ (i,j) \ : \ \alpha_i+\alpha_j = \alpha \}$.
Equating the coefficients of $p$ and $z^TQz$ yields the following
linear equality constraints on the entries of $Q$:
\begin{align}
\label{eq:Qeqs}
\sum_{(i,j)\in S_\alpha} Q_{i,j} =  
\left\{
\begin{array}{lc}
c_\alpha & \alpha \in \cA \\
0        & \alpha \notin \cA
\end{array}
\right.
\end{align}
There exists $A\in\R^{l\times l_z^2}$ and $b\in \R^l$ such that these
equality constraints are given by $Aq=b$ \footnote{In addition to the
  equality constraints due to $p=z^TQz$ there are also equality
  constraints due to the symmetry condition $Q=Q^T$. Some solvers,
  e.g. Sedumi \cite{sedumi99}, internally handle these symmetry
  constraints.} where $q:=vec(Q)$ is the vector obtained by vertically
stacking the columns of $Q$. The dimension $l$ is equal to the number
of elements of $M+M$. 

The zero diagonal algorithm is based on two lemmas.

\vspace{0.1in}
\begin{lemma}
\label{lem:S2alpha}
If $S_{2\alpha_i} = \{(i,i)\}$ then
\begin{align}
  \label{eq:Qiieq}
  Q_{i,i} =
  \left\{
    \begin{array}{lc}
      c_{2\alpha_i} & 2\alpha_i \in \cA \\
      0        & 2\alpha_i \notin \cA
    \end{array}
  \right.
\end{align}
\end{lemma}

\vspace{0.1in}
\begin{lemma}  
\label{lem:psos}
  If $p=z^TQz$, $Q\succeq 0$, and $Q_{i,i}=0$ then $p=\tilde{z}^T
  \tilde{Q} \tilde{z}$ where $\tilde{z}$ is the $(l_z-1) \times 1$
  vector obtained by deleting the $i^{th}$ element of $z$ and
  $\tilde{Q} \succeq 0$ is the $(l_z-1)\times (l_z-1)$ matrix obtained
  by deleting the $i^{th}$ row and column from $Q$.
\end{lemma}
\vspace{0.1in}

Lemma~\ref{lem:S2alpha} follows from Equation~\ref{eq:Qeqs}.
$S_{2\alpha_i} = \{(i,i)\}$ means that $x^{\alpha_i} \cdot
x^{\alpha_i}$ is the unique decomposition of $x^{2\alpha_i}$ as a
product of monomials in $z$.  There is no other decomposition of
$x^{2\alpha_i}$ as a product of monomials in $z$.  In this case,
$p=z^TQz$ places a direct constraint on $Q_{i,i}$ that must hold for
all possible Gram matrices.

Lemma~\ref{lem:psos} follows from a simple property of positive
semidefinite matrices: If $Q\succeq 0$ and $Q_{i,i} = 0$ then
$Q_{i,j}=Q_{j,i} = 0$ for $j=1,\ldots,l_z$. If $Q_{i,i}=0$ then
an SOS decomposition of $p$, if one exists, does not depend on the
monomial $z_i$ and $z_i$ can be removed from $z$. 

The zero diagonal algorithm is given in Table~\ref{tab:zda}.  The sets
$M_k$ denote the pruned list of monomial degree vectors at the
$k^{th}$ iterate.  The main step in the iteration is the search for
equations that directly constrain a diagonal entry $Q_{i,i}$ to be
zero (Step 6).  This step can be performed very fast since it can be
implemented using the $\texttt{find}$ command in Matlab.  Based on
Lemma~\ref{lem:psos}, if $Q_{i,i}=0$ then the monomial $z_i$ and the
$i^{th}$ row and column of $Q$ can be removed.  This is equivalent to
zeroing out the corresponding columns of $A$ (Step 7).  This
implementation has the advantage that $A$ and $b$ do not need to be
recomputed for each updated set $M_k$.  Zeroing out columns of $A$ in
Step 7 also means that new equations of the form $Q_{i,i}=0$ may be
uncovered during the next iteration.  The iteration continues until no
new zero diagonal entries of $Q$ are discovered.  The next theorem
proves that if $p$ is a SOS then the decomposition must be expressible
using only monomials associated with the final set $M_{k_f}$.
Moreover, $M_{k_f} \subseteq \frac{1}{2}C(p) \cap \N^n$, i.e.  the
list of monomials returned by the zero diagonal algorithm is never
larger than the list obtained from the Newton polytope method.  In
fact, there are polynomials for which the zero diagonal algorithm
returns a strictly smaller list of monomials than the Newton polytope.
The second example below provides an instance of this fact.

\begin{table}[h]
\begin{tabbing}
1. \= {\tt Given:} \= A polynomial $p = \sum_{\alpha \in \cA} 
c_\alpha x^\alpha$. \\
2. \> {\tt Initialization:} Set $k=0$ and 
   $M_0:=\{\alpha_i\}_{i=1}^{l_z} \subseteq \N^n$ \\
3. \> {\tt Form $Aq=b$:} Construct the equality constraint
   data, $A\in\R^{l\times l_z^2}$ and \\
   \>\> $b\in\R^l$, obtained by
   equating coefficients of $p=z^TQz$. \\
4. \> {\tt Iteration:} \\
5. \> \> Set $\cZ = \emptyset$, $k:=k+1$, and $M_k:=M_{k-1}$ \\
6. \> \> Search $Aq=b$: If there is an  equation of the form 
     $Q_{i,i} = 0$ \\
   \>\> then set $M_k:=$ $M_k\backslash \{\alpha_i\}$ and 
     $\cZ = \cZ \cup \cI$ where $\cI$ are the \\
   \>\> entries of $q$ corresponding to the $i^{th}$ row and column of $Q$. \\
7. \> \> For each $j\in \cZ$ set the $j^{th}$ column of $A$ equal to
    zero.   \\
8. \> \> Terminate if $\cZ = \emptyset$ otherwise return to step 5.\\
9. \> {\tt Return:}\ $M_k$, $A$, $b$
\end{tabbing}
\caption{Monomial Reduction using the Zero Diagonal Algorithm}
\label{tab:zda}
\vspace{-0.3in}
\end{table}

\vspace{0.1in}
\begin{theorem}
\label{thm:zda}
The zero diagonal algorithm terminates in a finite number of steps,
$k_f$, and $M_{k_f} \subseteq \frac{1}{2}C(p) \cap \N^n$. Moreover,
if $p = \sum_{i=1}^m f_i^2$ then $C(f_i) \cap \N^n \subseteq M_{k_f}$.
\end{theorem}
\begin{proof}
  $M_0$ has $l_z$ elements. The algorithm terminates unless at
  least one point is removed from $M_k$. Thus the algorithm must
  terminate after $k_f \le l_z+1$ steps.
                                 
  To show $M_{k_f} \subseteq \frac{1}{2}C(p) \cap \N^n$ consider a
  vertex $\alpha_i$ of $\ch(M_{k_f})$.   If there exists $u,v \in
  \ch(M_{k_f})$ such that $2\alpha_i = u+v$ then $u=v=\alpha_i$.  This
  follows from $\alpha_i = \frac{1}{2} (u+v)$ and the definition of a
  vertex. As a consequence, $S_{2\alpha_i} = \{ (i,i) \}$. By
  Lemma~\ref{lem:S2alpha}
  \begin{align}  
    Q_{i,i} =
    \left\{
      \begin{array}{lc}
        c_{2\alpha_i} & 2\alpha_i \in \cA \\
        0        & 2\alpha_i \notin \cA
      \end{array}
    \right.
  \end{align}
  $Q_{i,i} \ne 0$ since $\alpha_i$ was not removed at step 6 during
  the final iteration and thus $2\alpha_i \in \cA \subseteq C(p)$.
  This implies that $\alpha_i \in \frac{1}{2} C(p)$, i.e.
  $\frac{1}{2} C(p)$ contains all vertices of $\ch(M_{k_f})$. Hence
  $M_{k_f} \subseteq \ch(M_{k_f}) \subseteq \frac{1}{2}C(p)$.
  
  Finally it is shown that $C(f_i) \cap \N^n \subseteq M_{k_f}$.
  $C(f_i) \subseteq \frac{1}{2} C(p)$ by Theorem~\ref{thm:newton} and
  $\frac{1}{2}C(p) \subseteq \ch(M_0)$ by the choice of $M_0$.  Thus
  $C(f_i) \cap \N^n \subseteq M_0$.  Let $z$ be the vector of
  monomials associated with $M_0$.  If $p = \sum_{i=1}^m f_i^2$ then
  there exists a $Q\succeq 0$ such that $p=z^TQz$.  If the iteration
  removes no degree vectors then $M_{k_f} = M_0$ and the proof is
  complete.  Assume the iteration removes at least one degree vector
  and let $\alpha_i$ be the first removed degree vector.  Based on
  Step 6, $p=z^TQz$ constrains $Q_{i,i}=0$. By Lemma~\ref{lem:psos}
  the monomial $z_i$ cannot appear in any $f_i$.  Hence $C(f_i) \cap
  \N^n \subseteq M_{0} \backslash \{ \alpha_i \}$.  Induction can be
  used to show $C(f_i) \cap \N^n \subseteq M_{k}$ holds after each
  step $k$ including the final step $k_f$.
\end{proof}
\vspace{0.1in}

This algorithm is currently implemented in SOSOPT
\cite{nasaworkshop09}.  The results in Theorem~\ref{thm:zda} still
hold if $M_0\subseteq \N^n$ is chosen to be any set satisfying
$\frac{1}{2} C(p) \cap \N^n \subseteq M_0$.  Simple heuristics can be
used to obtain an initial set of monomials $M_0$ with fewer than
$l_z$ elements.  $M_0$ can then be used to initialize the zero diagonal
algorithm.  The next step is to construct the matrix $A$ and vector
$b$ obtained by equating the coefficients of $p$ and $z^TQz$.  This
step is required to formulate the LMI feasibility problem and it is
not an additional computational cost associated with the zero diagonal
algorithm.  $M_{k_f}$ contains the final reduced set of monomial degree
vectors.  If at least one degree vector was pruned then the returned
matrix $A$ and vector $b$ may contain entire columns or rows of zeros.
These rows/and columns can be deleted prior to passing the data to a a
semi-definite programming solver.  The next two examples demonstrate
the basic ideas of the algorithm.

\vspace{0.1in}
\underline{\emph{Example:}} Consider again the polynomial in
Equation~\ref{eq:polyex}. The full list of all monomials in two
variables with degree $\le 2$ consists of six monomials
(Equation~\ref{eq:fullz}).  Equating the coefficients of $p$ and
$z^TQz$ yields the following linear equality constraints on the
entries of $Q$: 
{\small
\begin{align*}
Q_{2,1} + Q_{1,2} &=  0, & Q_{4,1} + Q_{1,4} + Q_{2,2} &=  7 \\
Q_{4,2} + Q_{2,4} &=  0, & Q_{6,4} + Q_{4,6} + Q_{5,5} &=  0 \\
Q_{3,1} + Q_{1,3} &=  0, & Q_{6,1} + Q_{1,6} + Q_{3,3} &=  4 \\
Q_{5,4} + Q_{4,5} &=  0,  & Q_{5,2} + Q_{2,5} + Q_{4,3} + Q_{3,4} &= -2 \\
Q_{6,3} + Q_{3,6} &=  0,  & Q_{6,2} + Q_{2,6} + Q_{5,3} + Q_{3,5} &=  0 \\
Q_{6,5} + Q_{5,6} &=  0, & Q_{5,1} + Q_{1,5} + Q_{3,2} + Q_{2,3} &= -4 \\
Q_{1,1} & = 1     & Q_{4,4} &=  3 \\  
Q_{6,6} &=  0 
\end{align*}}
A matrix $A$ and vector $b$ can be constructed to represent these
equations in the form $Aq=b$.  Note that $Q_{6,6}=0$ and this implies
that $Q_{i,6} = Q_{6,i}=0$ $i=1,\ldots,6$ for any SOS decomposition of
$p$.  Thus the monomial $z_6 = x_2^2$ can not appear in any SOS
decomposition and it can be removed from the list.  After eliminating
$x_2^2$ and removing the $6^{th}$ row and column of $Q$, the equality
constraints reduce to:
{\small
\begin{align*}
Q_{2,1} + Q_{1,2} &=  0, & Q_{4,1} + Q_{1,4} + Q_{2,2} &=  7 \\
Q_{4,2} + Q_{2,4} &=  0, & Q_{5,5} &=  0 \\
Q_{3,1} + Q_{1,3} &=  0, & Q_{3,3} &=  4 \\
Q_{5,4} + Q_{4,5} &=  0,  & Q_{5,2} + Q_{2,5} + Q_{4,3} + Q_{3,4} &= -2 \\
Q_{5,3} + Q_{3,5} &=  0  & Q_{5,1} + Q_{1,5} + Q_{3,2} + Q_{2,3} &= -4 \\
Q_{1,1} & = 1 & Q_{4,4} &=  3 
\end{align*}}
Removing the $6^{th}$ row and column of $Q$ is equivalent to zeroing
out the appropriate columns of the matrix $A$.  This uncovers the new
constraint $Q_{5,5} = 0$ which implies that the monomial $z_5 =
x_1x_2$ can be pruned from the list. After eliminating $x_1x_2$, the
procedure can be repeated once again after removing the $5^{th}$ row
and column of $Q$.  No new diagonal entries of $Q$ are constrained to
be zero and hence no additional monomials can be pruned from $z$.  The
final list of monomials consists of four monomials.
\begin{align}
\label{eq:z4}
z = \bmtx 1 & x_1 & x_2 & x_1^2 \emtx^T
\end{align}
The Newton polytope method returned the same list.

\vspace{0.1in} 
\underline{\emph{Example:}} Consider the polynomial $p=x_1^2 + x_2^2 +
x_1^4 x_2^4$.  The Newton polytope is $C(p) = \ch( \{ \bsmtx 2 \\ 0
\esmtx, \ \bsmtx 0 \\ 2 \esmtx, \ \bsmtx 4 \\ 4 \esmtx \} )$. The
reduced Newton polytope is $\frac{1}{2} C(p) = \ch( \{ \bsmtx 1 \\ 0
\esmtx, \ \bsmtx 0 \\ 1 \esmtx, \ \bsmtx 2 \\ 2 \esmtx \} )$. The
monomial vector corresponding to $\frac{1}{2}C(p) \cap \N^n$ is:
\begin{align}
\label{eq:z4NP}
z:=\bmtx x_1, \ x_2, \ x_1 x_2, \ x_1^2 x_2^2 \emtx^T
\end{align}
There are $l_z = 15$ monomials in two variables with degree $\le 4$.
For simplicity, assume the zero diagonal algorithm is initialized with
$M_0 := \frac{1}{2} C(p) \cap \N^n$.  Equating coefficients of $p$ and
$z^TQz$ yields the constraint $Q_{3,3} = 0$ in the first iteration of
the zero diagonal algorithm.  The monomial $z_3 = x_1 x_2$ is pruned
and no additional monomials are removed at the next iteration. The
zero diagonal algorithm returns $M_2 = \{ \bsmtx 1 \\ 0 \esmtx, \ 
\bsmtx 0 \\ 1 \esmtx, \ \ \bsmtx 2 \\ 2 \esmtx \} $.  $M_2$ is a
proper subset of $\frac{1}{2} C(p) \cap \N^n$.  The same set of
monomials is returned by the zero diagonal algorithm after 13 steps if
$M_0$ is initialized with the $l_z=15$ degree vectors corresponding to
all possible monomials in two variables with degree $\le 4$. This
example demonstrates that the zero diagonal algorithm can return a
strictly smaller set of monomials than the Newton polytope method.

%-------------------------------------------------------
% 
%-------------------------------------------------------
\section{Simplification Method for SOS Programs}
\label{sec:sossimp}

This section describes a simplification method for SOS programs that is
based on the zero diagonal algorithm.  A sum-of-squares program is an
optimization problem with a linear cost and affine SOS constraints on
the decision variables \cite{sostools04}:
\begin{align}
\label{eq:sosprog}
& \min_{\dv\in\R^r} c^T \dv\\
\nonumber
& \mbox{subject to: } a_k(x,\dv) \in \sos, \ \ k=1,\ldots N
\end{align}
$\dv\in \R^r$ are decision variables.  The polynomials $\{ a_{k}
\}_{k=1}^N$ are given problem data and are affine in $\dv$:
\begin{align}
a_k(x,\dv):=a_{k,0}(x) + a_{k,1}(x)\dv_1 + \dots + a_{k,r}(x)\dv_r
\end{align}

Theorem~\ref{thm:gram} is used to convert an SOS program into a
semidefinite program (SDP).  The constraint $a_{k}(x,\dv) \in \sos$
can be equivalently written as:
\begin{align} 
\label{eq:sos_eq} 
a_{k,0}(x) + a_{k,1}(x)\dv_1 + \dots + a_{k,r}(x)\dv_r = z_k^T Q_k z_k \\ 
\label{eq:sos_lmi} 
Q_k \succeq 0
\end{align}
If $\max_\dv [\deg a_k(x,\dv)] = 2d$ then, in general, $z_k$ must
contain all monomials in $n$ variables of degree $\le d$.  $Q_k$ is a
new matrix of decision variables that is introduced when converting an
SOS constraint to an LMI constraint.  Equating the coefficients of
$z_k^TQ_kz_k$ and $a_{k}(x,\dv)$ imposes linear equality constraints
on the decision variables $\dv$ and $Q_k$.  There exists a matrix
$A\in\R^{l\times m}$ and vector $b\in\R^l$ such that the linear
equations for all SOS constraints are given by $Ay=b$ where
\begin{align}
\label{eq:ysdp}
  y:=[\dv^T, \ vec(Q_1)^T, \ \ldots, \ vec(Q_N)^T]^T
\end{align}
$vec(Q_k)$ denotes the vector obtained by vertically stacking the
columns of $Q_k$.  The dimension $m$ is equal to $r+\sum_{k=1}^N
m_k^2$ where $Q_k$ is $m_k \times m_k$ ($k=1,\ldots,N$).  After
introducing a Gram matrix for each constraint the SOS program can be
expressed as:
\begin{align}
\label{eq:sossdp}
& \min_{\dv \in \R^r,\{Q_k\}_{k=1}^N} c^T \dv\\
\nonumber
& \mbox{subject to: } A y = b \\
\nonumber
& Q_k \succeq 0, \ \ k=1,\ldots N
\end{align}
Equation~\ref{eq:sossdp} is an SDP expressed in Sedumi \cite{sedumi99}
primal form.  $\dv$ is a vector of free decision variables and
$\{Q_k\}_{k=1}^N$ contain decision variables that are constrained to
lie in the positive semi-definite cone. Sedumi internally handles the
symmetry constraints implied by $Q_k=Q_k^T$.

The SOS simplification procedure is a generalization of the zero
diagonal algorithm. It prunes the list of monomials used in each SOS
constraint.  It also attempts to remove free decision variables that
are implicitly constrained to be zero. Specifically, the constraints
in some SOS programs imply both $\dv_i \ge 0$ and $\dv_i \le 0$, i.e.
there is an implicit constraint that $\dv_i = 0$ for some $i$.
Appendix A.1 of \cite{tan06} provides some simple examples of how
these implicit constraints can arise in nonlinear analysis problems.
For these simple examples it is possible to discover the implicit
constraints by examination. For larger, more complicated analysis
problems it can be difficult to detect that implicit constraints
exist. The SOS simplification procedure described below automatically
uncovers some classes of implicit constraints $\dv_i=0$ and removes
these decision variables from the optimization.  This is important
because implicit constraints can cause numerical issues for SDP
solvers.  A significant reduction in computation time and improvement
in numerical accuracy has been observed when implicitly constrained
variables are removed prior to calling Sedumi.

The general SOS simplification procedure is shown in
Table~\ref{tab:sossimp}. To ease the notation the algorithm is only
shown for the case of one SOS constraint ($N=1$).  The extension to
SOS programs with multiple constraints ($N>1$) is straight-forward.
The algorithm is initialized with a finite set of vectors $M_0
\subseteq \N^n$. The Newton polytope of $a(x,\dv)$ depends on the
choice of $\dv$ so $M_0$ must be chosen so that it contains all
possible reduced Newton polytopes. One choice is to initialize $M_0$
corresponding to the degree vectors of all monomials in $n$ variables
and degree $\le 2d:=\max_\dv [\deg a_k(x,\dv)]$.  $A$ and $b$ need to
be computed when formulating the SDP so this step is not additional
computation associated with the simplification procedure.  The last
pre-processing step is the initialization of the sign vector $s$.  The
entries of $s_i$ are $+1$, $-1$, or $0$ if it can be determined from
the constraints that $y_i$ is $\ge 0$, $\le 0$ or $=0$, respectively.
$s_i=$\texttt{NaN} if no sign information can be determined for $y_i$.
If $y_i$ corresponds to a diagonal entry of $Q$ then $s_i$ can be
initialized to $+1$.

The main iteration step is the search for equations that directly
constrain any decision variable to be zero (Step 7a).  This is similar
to the zero diagonal algorithm.  The iteration also attempts to
determine sign information about the decision variables.  Steps 7b-7d
update the sign vector based on equality constraints involving a
single decision variable.  For example, a decision variable must be
zero if the decision variable has been previously determined to be
$\le 0$ and the current equality constraint implies that it must be
$\ge 0$ (Step 7c).  These decision variables can be removed from the
optimization.  Step 8 processes equality constraints involving two
decision variables.  The logic for this case is omitted due to space
constraints.  The processing of equality constraints can be performed
very fast using the $\texttt{find}$ command in Matlab.  Steps 9 and 10
prune monomials and zero out appropriate columns of $A$.  The
iteration continues until no additional information can be determined
about the sign of the decision variables.

\begin{table}[h]
\begin{tabbing}
1. \= {\tt Given:} \= Polynomials $\{ a_j \}_{j=1}^r$ in variables
   $x$.  Define \\
   \>\> $a(x,\dv):= a_0(x) + a_1(x)\dv_1 + \dots + a_r(x)\dv_r$ \\
2. \> {\tt Initialization:} Set $k=0$ and choose a finite set 
   $M_0:=\{\alpha_i\}_{i=1}^{m}$ \\
   \> \> $\subseteq \N^n$ such that 
   $\left[ \cup_{\dv \in \R^r} \frac{1}{2} C(a(x,\dv)) \right] \cap \N^n 
   \subseteq M_0$. \\
3. \> {\tt Form $Ay=b$:} Construct the equality constraint data,
   $A\in\R^{l\times (r+m^2)}$\\
   \>\> and $b\in\R^l$ obtained by equating coefficients of $a(x,\dv)=z^TQz$ \\
   \>\> where $z:=\bmtx x^{\alpha_1},\ldots, x^{\alpha_m} \emtx^T$ 
   and $y:=[\dv^T, \ vec(Q)^T]^T$. \\
4. \> {\tt Sign Data:} Initialize the $l\times 1$ vector $s$ 
    to be $s_i = +1$ if $y_i$ \\
   \> \> corresponds to a diagonal entry of $Q$.  Otherwise set 
   $s_i =$ \texttt{NaN}. \\
5. \> {\tt Iteration:} \\
6. \> \> Set $\cZ = \emptyset$, $\cS = \emptyset$, $k:=k+1$, and
    $M_k:=M_{k-1}$ \\
7. \> \> Pro\=cess equality equality constraints of the form 
          $a_{i,j} y_j = b_i$ \\
   \> \> where $a_{i,j} \ne 0$\\
\ \ 7a. \> \>  \> If $b_i=0$ then set $s_j=0$ and $\cZ = \cZ \cup j$\\ 
\ \ 7b. \> \>  \> Else if $s_j=$\texttt{NaN} then set 
        $s_j =$ sign$(a_{i,j} b_i)$ and $\cS = \cS \cup j$\\ 
\ \ 7c. \> \>  \> Else if $s_j=-1$ and sign$(a_{i,j} b_i)=+1$ then
         set $s_j = 0$ \\
        \> \>\>  and $\cS = \cS \cup j$\\ 
\ \ 7d. \> \>  \> Else if $s_j=+1$ and sign$(a_{i,j} b_i)=-1$ then
         set $s_j = 0$ \\
        \> \>\>  and $\cS = \cS \cup j$ \\
8. \> \> Process equality equality constraints of the form \\
   \>\>       $a_{i,j_1} y_{j_1} + a_{i,j_2} y_{j_2} = b_i$. \\
9. \> \> If for any $j\in \cZ$, $y_j$ corresponds to a diagonal entry 
    $Q_{i,i}$ \\
   \>\> then set $M_k := M_k \backslash \{ \alpha_i\}$ and  
    $\cZ = \cZ \cup \cI$ where $\cI$ are the \\
   \>\> entries of $y$ corresponding to the $i^{th}$ row and column of 
   $Q$.\\
10. \> \> For each $j\in \cZ$ set the $j^{th}$ column of $A$ equal to
    zero.   \\
11. \> \> Terminate if $\cZ = \emptyset$ and $\cS = \emptyset$
          otherwise return to step 6.\\
12. \> \ \ {\tt Return:}\ $M_k$, $A$, $b$, $s$
\end{tabbing}
\caption{Simplification Method for SOS Programs with One Constraint}
\label{tab:sossimp}
\vspace{-0.3in}
\end{table}

%-------------------------------------------------------
% Conclusions
%-------------------------------------------------------
\section{Conclusions}
\label{sec:conclusions}

The Newton polytope is a method to prune unnecessary monomials from an
SOS decomposition.  The method requires the construction of a convex
hull and this can be time time consuming for polynomials with many
terms. This paper presented a zero diagonal algorithm for pruning
monomials. The algorithm is based on a simple property of positive
semidefinite matrices. The algorithm is fast since it only requires
searching a set of linear equality constraints for those having
certain properties.  Moreover, the set of monomials returned by the
algorithm is a subset of the set returned by the Newton polytope
method.  The zero diagonal algorithm was extended to a more general
reduction method for sum-of-squares programming.

%-------------------------------------------------------
% Acknowledgments
%-------------------------------------------------------
\section{Acknowledgments}

This research was partially supported under the NASA Langley NRA
contract NNH077ZEA001N entitled ``Analytical Validation Tools for Safety
Critical Systems'' and the NASA Langley NNX08AC65A contract 
entitled 'Fault Diagnosis, Prognosis and Reliable Flight 
Envelope Assessment.'' The technical contract monitors are Dr. Christine
Belcastro and Dr. Suresh Joshi respectively.

%-------------------------------------------------------
% REFERENCES (IN .BIB FILE)             
%-------------------------------------------------------
\bibliography{sossimplify}

\end{document}